%% file: resubmit.tex
\documentclass[oneside]{amsart}

\usepackage[top=.8in,bottom=1in,left=1in,right=1in,footskip=.5in]{geometry}
\usepackage{microtype}
\usepackage{mathtools}
\usepackage{hyperref}
\usepackage[dvipsnames]{xcolor}
\usepackage{tikz}
\usetikzlibrary{arrows}
\usepackage{lpic}
\usepackage{caption}
\usepackage{footnote}

\title{Algebraic $k$-systems of curves}
\author{Charles Daly, Jonah Gaster, Max Lahn, Aisha Mechery, Simran Nayak}

\date{February 13, 2020}

\address{Department of Mathematics, University of Maryland \\}
\email{cdaly69@math.umd.edu\\}

\address{Department of Mathematics, University of Wisconsin-Milwaukee \\}
\email{gaster@uwm.edu\\}

\address{Department of Mathematics, University of Michigan \\}
\email{maxlahn@umich.edu\\}

\address{Department of Mathematics, Bryn Mawr College \\}
\email{aisha.mechery@gmail.com\\}

\address{Department of Applied Mathematics, Brown University \\}
\email{simran\_nayak@brown.edu\\}

\theoremstyle{plain}
\newtheorem{theorem}{Theorem}
\newtheorem{proposition}[theorem]{Proposition}
\newtheorem{corollary}[theorem]{Corollary}

\newtheorem*{theorem*}{Theorem}
\newtheorem*{proposition*}{Proposition}
\newtheorem*{corollary*}{Corollary}
\newtheorem*{axiom*}{Axiom}
\newtheorem{lemma}[theorem]{Lemma}
\newtheorem*{lemma*}{Lemma}

\theoremstyle{definition}

\newtheorem{definition/proposition}[theorem]{Definition/Proposition}
\newtheorem*{definition*}{Definition}
\newtheorem*{definitions*}{Definitions}
\newtheorem*{example*}{Example}
\newtheorem*{examples*}{Examples}
\newtheorem*{exercise*}{Exercise}
\newtheorem*{exercises*}{Exercises}
\newtheorem*{definition/proposition*}{Definition/Proposition}

\newtheorem*{problem*}{Problem}

\theoremstyle{remark}
\newtheorem{remark}[theorem]{Remark}

\newtheorem*{remark*}{Remark}
\newtheorem*{remarks*}{Remarks}
\newtheorem*{note*}{Note}
\newtheorem*{observation*}{Observation}
\newtheorem*{recall*}{Recall}
\newtheorem*{notation*}{Notation}
\newtheorem*{terminology*}{Terminology}
\newtheorem*{fact*}{Fact}
\newtheorem*{conjecture*}{Conjecture}
\newtheorem{construction}{Construction}

\newcommand{\Z}{\mathbb{Z}}
\newcommand{\Hom}{\mathrm{H}}
\newcommand{\Sp}{\mathrm{Sp}}
\newcommand{\SL}{\mathrm{SL}}

\newcommand{\vect}{\mathbf}

\begin{document}

\begin{abstract}
A collection $ \Delta $ of simple closed curves on an orientable surface is an algebraic $ k $-system if the algebraic intersection number $ \langle \alpha, \beta \rangle$ is equal to $k $ in absolute value for every $ \alpha , \beta \in \Delta $. 
Generalizing a theorem of \cite{MRT14} we compute that the maximum size of an algebraic $k$-system of curves on a surface of genus $g$ is $2g+1$ when $g\ge 3$ or $k$ is odd, and $2g$ otherwise. 
To illustrate the tightness in our assumptions, we present a construction of curves pairwise geometrically intersecting twice whose size grows as $g^2$.
\end{abstract}

\maketitle

\section{Introduction} \label{sec:intro}

Questions about the combinatorial properties of collections of curves on surfaces with specified pairwise intersection data have been considered by various authors.
In \cite{JMM96}, the first explicit bounds were obtained for the maximum size of any curve collection on a surface of finite-type such that any pair has geometric intersection number at most $ k $. 
In the remarkably delicate case $ k = 1 $, increasingly precise bounds on these maximum sizes are given in \cite{MRT14}, \cite{Prz15}, \cite{ABG17}, and \cite{Gre18a}, but aside from several low-complexity examples, exact values remain unknown.

In \cite{MRT14} it is shown that the maximum size of a collection of simple curves on a connected oriented surface $ S $, of genus $g$, such that every pair of distinct curves intersects an odd number of times is $ 2 g + 1 $.
The authors' proof exploits the symplectic pairing on $ \Hom_1( S ; \Z / 2 \Z ) $ induced by the algebraic intersection number. 
Our main theorem is an extension of this result. 
The proof is slightly different than that of \cite{MRT14}, using the symplectic pairing of vectors to apply induction on $g$. 
This allows a bit more information.  
We say that a subset $\Lambda \subset \Z^{2g}$ is a \emph{symplectic $k$-system} if $|\langle \vect{v},\vect{w} \rangle|=k$ for all distinct $\vect v,\vect w\in\Lambda$.

\begin{theorem} \label{thm:curves}
Let $k \equiv 2^{ m - 1 } \pmod{ 2^{ m } } $.
A symplectic $k$-system has size at most $2g+1$, and equality is realized. 
When $g\ge 3$ or $m=1$, there exist primitive symplectic $k$-systems of size $2g+1$. 
When $m>1$ and $g\le 2$, a primitive symplectic $k$-system has size at most $2g$, and equality is realized.
\end{theorem}

Since the geometric and algebraic intersection numbers of any pair of curves on $ S $ have the same parity, \cite[Thm~1.4]{MRT14} is the case $ m = 1 $ above. 
Naturally, curves determine primitive homology classes via the isomorphism $\Hom_1( S ; \Z  )  \cong  \Z ^{ 2 g } $.

\begin{corollary} \label{cor:systems}
The maximum size of an algebraic $k$-system of simple curves on $ S $ is $ 2 g + 1 $ when $g\ge 3$. For $g\le 2$, the latter is $2g$ when $k$ is even and $2g+1$ when $k$ is odd.
\end{corollary}

On the other hand, to emphasize the importance of `algebraic' above we show:

\begin{proposition*} \label{prop:geom 2system}
There is a collection of simple closed curves on $S$ of size $\approx g^2$ that pairwise geometrically intersects twice. 
\end{proposition*}

While this paper was being composed, a short proof of the upper bound of $2g+1$ was found in conversation with Josh Greene. This amounts to a more direct generalization of the $m=1$ argument in \cite[Prop.~5.1]{MRT14}. 
We include both this short argument in Section~\ref{sec:simple proof} and our inductive proof in Section~\ref{sec:vectors}, which, while slightly longer, allows greater control in small genus.

\bigskip

\subsection*{Acknowledgements}
The authors thank Tarik Aougab, Moira Chas, and Bill Goldman for their enthusiasm and support.
We are grateful as well to Josh Greene for his help with an alternative proof of 
part of our main theorem (see Proposition~\ref{prop:Josh}). 
This material is based upon work supported by the National Science Foundation under Grant No.~DMS-1439786 while the authors were in residence at the Institute for Computational and Experimental Research in Mathematics in Providence, RI, during the \href{https://icerm.brown.edu/summerug/2018/}{\emph{Summer@ICERM 2018: Low Dimensional Topology and Geometry}} program.

\newpage
\section{Diagonal Orbits of Symplectic Groups} \label{sec:orbits}

In this section, we adapt the following result of Athreya and Konstantoulas from \cite{AK18} to obtain simple representatives for the orbits of the diagonal action of the symplectic group $ \Sp( 2 g ; \Z / 2^{ m } \Z ) $ on pairs of elements in $ \left( \Z / 2^{ m } \Z \right)^{ 2 g } $. For this section only, we assume that $ g \geq 2 $.

\begin{proposition}
\cite[Prop.~3.5]{AK18}\label{prop:orbits}
The orbits of the diagonal action of the integral symplectic group $ \Sp( 2 g ; \Z ) $ on pairs of primitive elements in $\textstyle \Z^{ 2 g } $ are in correspondence with ordered triples $ ( a , b , c )$ of natural numbers, where either
\begin{enumerate}

\item[(a)]

$ c = 0 $, in which case $ 0 < a < b $ and $ \gcd( a , b ) = 1 $;

\item[(b)]

$ c = 1 $, in which case $ a = 0 $; or

\item[(c)]

$ c > 1 $, in which case $ c $ divides $ b $, $ 0 < a < c $, and $ \gcd( a , c ) = 1 $.

\end{enumerate}
Moreover, $ (\vect{ e }_{ 1 } , a \vect{ e }_{ 1 } + b \vect{ e }_{ g + 1 } + c \vect{ e }_{ 2 g } ) $ is a representative for the orbit corresponding to $ ( a , b , c ) $.

\end{proposition}

\begin{corollary} \label{cor:pairs}

For any pair $ \vect{ u } , \vect{ v } \in \left( \Z / 2^{ m} \Z  \right)^{ 2 g } $ with $\vect{u}$ primitive, there exists $ T \in \Sp( 2 g ; \Z / 2^{ m } \Z ) $ taking $ \vect{ u } $ to $ \vect{ e }_{ 1 } $ and $ \vect{ v } $ to $ y \vect{ e }_{ 1 } + \langle  \vect{ u } , \vect{ v } \rangle \vect{ e }_{ g + 1 } + z \vect{ e }_{ 2 g } $ for some $ y $ and $ z $ in $ \Z / 2^{ m } \Z $.

\begin{proof}

Write $ \vect{ v } = \lambda \vect{ w } $, where $ \vect{ w } $ is primitive, and lift $ \vect{ u } $ and $ \vect{ w } $ to primitive elements in $ \Z^{ 2 g } $. 
By Proposition~\ref{prop:orbits}, there exists a transformation $ \tilde{ T } $ in $ \Sp( 2 g ; \Z ) $ taking $ \tilde{ \vect{ u } } $ to $ \vect{ e }_{ 1 } $ and $ \tilde{ \vect{ w } } $ to $ a \vect{ e }_{ 1 } + b \vect{ e }_{ g + 1 } + c \vect{ e }_{ 2 g } $ for some integers $ a $, $ b $, and $ c $. Note that
\[
b \equiv \langle \vect{ e }_{ 1 } , a \vect{ e }_{ 1 } + b \vect{ e }_{ g + 1 } + c \vect{ e }_{ 2 g } \rangle \equiv \langle \tilde{ T } \tilde{ \vect{ u } } , \tilde{ T } \tilde{ \vect{ w } } \rangle \equiv \langle \tilde{ \vect{ u } } , \tilde{ \vect{ w } } \rangle \equiv \langle \vect{ u } , \vect{ w } \rangle \pmod{ 2^{ m } } .
\]
Therefore, the transformation $ T $ induced on $ \left( \Z / 2^{ m } \Z \right)^{ 2 g } $ by $ \tilde{ T } $ takes $ \vect{ u } $ to $ \vect{ e }_{ 1 } $ and $ \vect{ v } $ to
\[
T \vect{ v } = \lambda ( a \vect{ e }_{ 1 } + b \vect{ e }_{ g + 1 } + c \vect{ e }_{ 2 g } ) = \lambda a \vect{ e }_{ 1 } + \lambda \langle \vect{ u } , \vect{ w } \rangle \vect{ e }_{ g + 1 } + \lambda c \vect{ e }_{ 2 g } = y \vect{ e }_{ 1 } + \langle \vect{ u } , \vect{ v } \rangle \vect{ e }_{ g + 1 } + z \vect{ e }_{ 2 g } ,
\]
where $ y \coloneqq \lambda a $ and $ z \coloneqq \lambda c $.
\end{proof}

\end{corollary}

\bigskip
\section{Symplectic Systems} \label{sec:vectors}

We will call $ \Lambda \subset \left( \Z / 2^{ m } \Z \right)^{ 2 g } $ a \emph{symplectic system} if for any pair of distinct elements $ \vect{ u }, \vect{ v } \in \Lambda$ we have 
\[
\langle \vect{ u } , \vect{ v } \rangle \mathbin{ \overset{ \mathrm{ def } }{ \equiv } } \sum_{ i = 1 }^{ g }{ \det{ \begin{pmatrix} u_{ i } & v_{ i } \\ u_{ g + i } & v_{ g + i } \end{pmatrix} } } \equiv 2^{ m - 1 } \pmod{ 2^{ m } }~.
\]
In this case, we will call the natural numbers $ g $ and $ m $ the \emph{genus} and \emph{power} of $ \Lambda $, respectively. 

\begin{proposition} \label{prop:vectors}

A symplectic system in $ \left( \Z / 2^{ m } \Z \right)^{ 2 g } $ has size at most $ 2 g + 1 $.

\end{proposition}

We will prove this bound by induction first on the genus $ g \geq 0 $ and then on the power $ m \geq 1 $.
We separate the induction into two technical lemmas.
When a symplectic system contains primitive elements, we may obtain through Lemma~\ref{lem:primitive} a symplectic system of lower genus but comparable size. 
Conversely, when such a system contains no primitive elements, we obtain through Lemma~\ref{lem:non-primitive} a system of lower power and equal size.

\begin{lemma} \label{lem:primitive}

If $\Lambda$ is a symplectic system in $ \left( \Z / 2^{ m } \Z \right)^{ 2 g+2 } $ that contains at least one primitive element, there is a symplectic system $ \Lambda' $ in $ \left( \Z / 2^{ m } \Z \right)^{ 2 g } $ such that $ | \Lambda | \leq 2 + | \Lambda' | $.
\end{lemma}

Observe that Lemma~\ref{lem:primitive} applies when $m=1$: every nonzero element of $ \left( \Z / 2 \Z \right)^{ 2 g } $ is primitive.

\begin{proof}

Consider a symplectic system $\Lambda$ in $ \left( \Z / 2^{ m } \Z \right)^{ 2 g+2 } $ containing a primitive element $ \vect{ u } $, and first suppose that $ g = 0 $. 
Since $ \Lambda' = \{ \vect{ 0 } \} $ is trivially a symplectic system of genus $0$, it suffices to show that $ | \Lambda | \leq 2 + | \Lambda' | = 3 $. 
Since the symplectic group $ \Sp( 2 ; \Z / 2^m \Z ) = \SL( 2 ; \Z / 2^m \Z ) $ acts transitively on primitive elements, we may assume without loss of generality that $ \vect{ u } = \vect{ e }_{ 1 } $. 
Suppose for a contradiction that $ \Lambda $ contains at least three other elements $ \vect{ v } $, $ \vect{ w } $, and $ \vect{ x } $. Then
\[
v_{ 2 } \equiv \det{ \begin{pmatrix} 1 & v_{ 1 } \\ 0 & v_{ 2 } \end{pmatrix} } \equiv \langle \vect{ e }_{ 1 } , \vect{ v } \rangle \equiv \langle \vect{ u } , \vect{ v } \rangle \equiv 2^{ m - 1 } \pmod{ 2^{ m } } .
\]
For the same reasons, $ w_{ 2 } \equiv x_{ 2 } \equiv 2^{ m - 1 } \pmod{ 2^{ m } } $. Moreover,
\[
2^{ m - 1 } ( v_{ 1 } - w_{ 1 } ) \equiv \det{ \begin{pmatrix} v_{ 1 } & w_{ 1 } \\ 2^{ m - 1 } & 2^{ m - 1 } \end{pmatrix} } \equiv \langle \vect{ v } , \vect{ w } \rangle \equiv 2^{ m - 1 } \pmod{ 2^{ m - 1 } } ,
\]
and so $ v_{ 1 } - w_{ 1 } \equiv 1 \pmod{ 2 } $. For the same reasons, $ w_{ 1 } - x_{ 1 } \equiv x_{ 1 } - v_{ 1 } \equiv 1 \pmod{ 2 } $. However, this gives the contradiction
\[
0 \equiv ( v_{ 1 } - w_{ 1 } ) + ( w_{ 1 } - x_{ 1 } ) + ( x_{ 1 } - v_{ 1 } )\equiv 1 + 1 + 1 \equiv 1 \pmod{ 2 } .
\]
Now suppose that $ g $ is positive. 
If $\Lambda$ is a singleton then the conclusion is trivial, so we may assume that $ \Lambda $ contains elements $\vect{u}$ and $ \vect{ v } $. 
By Corollary~\ref{cor:pairs}, we may assume that $ \vect{ u } = \vect{ e }_{ 1 } $ and that $ \vect{ v } = y \vect{ e }_{ 1 } + 2^{ m - 1 } \vect{ e }_{ g + 2 }  + z \vect{ e }_{ 2g+2 } $ for some $ y $ and $ z $ in $ \Z / 2^{ m } \Z $. 
Let $ A \coloneqq \Lambda \setminus \{ \vect{ u } , \vect{ v } \} $. Then $ w_{ g + 2 } \equiv \langle \vect{ e }_{ 1 } , \vect{ w } \rangle \equiv \langle \vect{ u } , \vect{ w } \rangle \equiv 2^{ m - 1 } \pmod{ 2^{ m } } $ and
\[
2^{ m - 1 } ( y + w_{ 1 } ) - z w_{ g+1 } \equiv \langle y \vect{ e }_{ 1 } + 2^{ m - 1 } \vect{ e }_{ g + 2 } + z \vect{ e }_{ 2 g+2 } , \vect{ w } \rangle \equiv \langle \vect{ v } , \vect{ w } \rangle \equiv 2^{ m - 1 } \pmod{ 2^{ m } }
\]
for any element $ \vect{ w } $ in $ A $. 
For each such element $ \vect{ w } $ in $ A $, let $ \vect{ w }' \coloneqq \sum_{ i = 1 }^{ g } ( w_{ i + 1 } \vect{ e }_{ i } + w_{ g + i + 2 } \vect{ e }_{ g + i } ) + z \vect{ e }_{ 2 g } $. The above observations imply that, for any pair of elements $ \vect{ w } $ and $ \vect{ x } $ in $ A $,
\begin{align*}
\det{ \begin{pmatrix} w_{ g + 1 } & x_{ g + 1 } \\ w_{ 2 g + 2 } + z & x_{ 2 g + 2 } + z \end{pmatrix} } & \equiv w_{ g + 1 } z - x_{ g + 1 } z + \det{ \begin{pmatrix} w_{ g + 1 } & x_{ g + 1 } \\ w_{ 2 g + 2 } & x_{ 2 g + 2 } \end{pmatrix} } & \pmod{ 2^{ m } } \\
& \equiv 2^{ m - 1 } ( w_{ 1 } - x_{ 1 } ) + \det{ \begin{pmatrix} w_{ g + 1 } & x_{ g + 1 } \\ w_{ 2 g + 2 } & x_{ 2 g + 2 } \end{pmatrix} } & \vdots \qquad \\
& \equiv \det{ \begin{pmatrix} w_{ 1 } & x_{ 1 } \\ w_{ g + 2 } & x_{ g + 2 } \end{pmatrix} } + \det{ \begin{pmatrix} w_{ g + 1 } & x_{ g + 1 } \\ w_{ 2 g + 2 } & x_{ 2 g + 2 } \end{pmatrix} } & \pmod{ 2^{ m } }
\end{align*}
Therefore, $ \langle \vect{ w }' , \vect{ x }' \rangle \equiv \langle \vect{ w } , \vect{ x } \rangle \pmod{ 2^{ m } } $ for all such elements $ \vect{ w } $ and $ \vect{ x } $ in $ A $. 
In particular, this implies both that the map $ \vect{ w } \mapsto \vect{ w }' $ is injective, and that its image $ \Lambda' \coloneqq \{ \vect{ w }' : \vect{ w } \in A \} $ is a symplectic system in $ \left( \Z / 2^{ m } \Z \right)^{ 2 g } $ such that $ | \Lambda | = | \{ \vect{ u } , \vect{ v } \} | + | A | = 2 + | \Lambda' | $.
\end{proof}

\begin{lemma} \label{lem:non-primitive}

For every symplectic system $\Lambda $ in $ \left( \Z / 2^{ m+2 } \Z \right)^{ 2 g } $  containing no primitive elements, there is a symplectic system $\Lambda''$ in $ \left( \Z / 2^{ m } \Z \right)^{ 2g } $ such that $ | \Lambda | = | \Lambda'' | $.

\begin{proof}

Consider a symplectic system $\Lambda$ in $ \left( \Z / 2^{ m+2 } \Z \right)^{ 2 g } $ which contains no primitive elements. 
Since the units in $ \Z / 2^{ m + 2 } \Z $ are precisely the odd numbers, an element in $ \left( \Z / 2^{ m + 2 } \Z \right)^{ 2 g } $ is primitive if and only if it is not divisible by $ 2 $.

In particular, our hypothesis is that every element in $ \Lambda $ is divisible by $ 2 $. For each element $ \vect{ u } $ in $ \Lambda $, choose $ \vect{ u }' $ such that $ \vect{ u } = 2 \vect{ u }' $. 
Let $ \Lambda' \coloneqq \{ \vect{ u }' : \vect{ u } \in \Lambda \} $, so that $ | \Lambda | = | \Lambda' | $, and consider the quotient map $ q $ of $ \left( \Z / 2^{ m + 2 } \Z \right)^{ 2 g } $ onto $ \left( \Z / 2^{ m } \Z \right)^{ 2 g } $. Note that $ 4 \langle \vect{ u }' , \vect{ v }' \rangle \equiv \langle \vect{ u } , \vect{ v } \rangle \equiv 2^{ m + 1 } \pmod{ 2^{ m + 2 } } $ for any pair $ \vect{ u } , \vect{ v } \in \Lambda $, so that $ \langle q ( \vect{ u }' ), q ( \vect{ v }' ) \rangle \equiv \langle \vect{ u }' , \vect{ v }' \rangle \equiv 2^{ m - 1 } \pmod{ 2^{ m } } $ for any $ \vect{ u }' , \vect{ v }' \in \Lambda' $. 
This implies that $ \Lambda'' \coloneqq q ( \Lambda' ) $ is a symplectic system in $ \left( \Z / 2^{ m } \Z \right)^{ 2 g } $ such that $ | \Lambda | = | \Lambda' | = | \Lambda'' | $.

\end{proof}

\end{lemma}

Equipped with the above lemmas, we are now ready to prove Proposition~\ref{prop:vectors}.

\begin{proof}[Proof (of Proposition~\ref{prop:vectors})]

We induct first on the genus $ g $ and then on the power $ m $.

The case $g=0$ is trivial, since a subset of $ \left( \Z / 2^{ m } \Z \right)^{ 0 } = \{ \vect{ 0 } \} $ has size at most one. 
Now suppose that the desired bound holds for all symplectic systems of genus $ g $ (we refer to this as the `outer' inductive hypothesis below). 
In order to show that it holds for symplectic systems of genus $ g + 1 $, we proceed by induction on the power $ m $, where the increment of induction is $ 2 $.

Every nonzero vector in $(\Z/2\Z)^{2g}$ is primitive, so the base case $m=1$ follows from the inductive hypothesis for $g$ via Lemma~\ref{lem:primitive}.
In case $ m = 2 $, consider a symplectic system in $ \left( \Z / 4 \Z \right)^{ 2 g+ 2 } $. 
If $ \Lambda $ contains a primitive element, then Lemma~\ref{lem:primitive} gives a symplectic system $ \Lambda' $ in $ \left( \Z / 4 \Z \right)^{ 2 g } $ such that $ | \Lambda|  \leq 2 + | \Lambda' | $. 
In this case, $ | \Lambda | \leq 2 g + 3 $ by the outer inductive hypothesis. 
Conversely, note that if $ \Lambda $ contains two non-primitive elements $ 2 \vect{ u } $ and $ 2 \vect{ v } $, then
\[
0 \equiv 0 \cdot \langle \vect{ u } , \vect{ v } \rangle \equiv 4 \cdot \langle \vect{ u } , \vect{ v } \rangle \equiv \langle 2 \vect{ u } , 2 \vect{ v } \rangle \equiv 2 \pmod{ 4 } .
\]
Therefore, if $ \Lambda $ does not contain any primitive elements, then $ | \Lambda | \leq 1 < 2 g + 3 $.

Now suppose that the desired bound holds for all symplectic systems in $ \left( \Z / 2^{ m } \Z \right)^{ 2 g+2 } $ and consider a symplectic system $ \Lambda $ in $ \left( \Z / 2^{ m+2 } \Z \right)^{ 2 g+2 } $. 
As above, if $ \Lambda $ contains a primitive element, then Lemma~\ref{lem:primitive} gives a symplectic system $ \Lambda' $ in $ \left( \Z / 2^{ m+2 } \Z \right)^{ 2 g } $ such that $ | \Lambda | \leq 2 + | \Lambda' | $. 
In this case, $ | \Lambda | \leq 2 g + 3 $ by the outer inductive hypothesis. Else, if $ \Lambda $ contains no primitive vectors, then Lemma~\ref{lem:non-primitive} gives a symplectic system $ \Lambda'' $ in $ \left( \Z / 2^{ m } \Z \right)^{ 2 g+2 } $ such that $ | \Lambda | = | \Lambda'' |$. 
In this case, $ | \Lambda | \leq 2 g + 3 $ by the inner inductive hypothesis.
\end{proof}

By \cite[Thm.~1.4]{MRT14} there are primitive symplectic systems in $(\Z/2\Z)^{2g}$ of size $2g+1$, so Proposition~\ref{prop:vectors} is optimal when $m=1$.
However, in low genus, and for $m>1$, the proof of Lemma~\ref{lem:primitive} can yield slightly more. 

\begin{lemma}
\label{lem:low genus}
A primitive symplectic system $\Lambda\subset  \left( \Z / 2^{ m } \Z \right)^{ 2 g } $ with $m>1$ and $g\le 2$ has $|\Lambda|\le 2g$.
\end{lemma}

\begin{proof}
Let $g=1$.
Keeping the notation from the proof of Lemma~\ref{lem:primitive}, if $\vect{u},\vect{v},\vect{w}$ were a primitive symplectic system, then after applying an element of $\Sp(2;\Z/2^m\Z)$ we would have $v_2\equiv w_2\equiv 2^{m-1}$ while $v_1$ and $w_1$ are of opposite parity. 
As $m>1$, this implies that one of $\vect{v}$ or $\vect{w}$ is non-primitive.

For the case $g=2$, as in the proof of Lemma~\ref{lem:primitive} we assume that $\vect{u},\vect{v}\in \Lambda$ satisfy $ \vect{ u } = \vect{ e }_{ 1 } $ and $ \vect{ v } = y \vect{ e }_{ 1 } + 2^{ m - 1 } \vect{ e }_{ 3 }  + z \vect{ e }_{ 4 } $, for some $ y $ and $ z $ in $ \Z / 2^{ m } \Z $. The remaining elements may be split into two sets $\Lambda\setminus \{\vect{u},\vect{v}\} = \Lambda_e \cup \Lambda_o$, where
\[
\Lambda_e = \{ \vect{w}\in\Lambda \setminus \{ \vect{u},\vect{v} \}:w_1 \text{ is even } \} , \text{ and } \Lambda_o = \{ \vect{w} \in \Lambda  \setminus \{ \vect{u},\vect{v} \} : w_1 \text{ is odd } \}~.
\]
Consider the projection $\pi:( \Z / 2^{ m } \Z )^4 \to  (\Z / 2^{ m } \Z )^2$ sending $\vect e_1$ and $\vect e_3$ to $\vect 0$, and let $\mathcal{E} =\pi( \Lambda_e)$, $\mathcal{O}=\pi(\Lambda_o)$, and $\vect{z} = \pi(\vect{v})$.
The desired conclusion will follow from $|\mathcal{E}|+|\mathcal{O}|\le 2$.

First we observe that $\mathcal{E}$ and $\mathcal{O}$ are symplectic systems, and hence $|\mathcal{E}|=|\Lambda_e|$ and $|\mathcal{O}|=|\Lambda_o|$. 
Indeed, if $\vect{w},\vect{w}'\in\Lambda_e$, then $\langle \vect{w},\vect{w}' \rangle \equiv w_2w_4'-w_4w_2' \equiv \langle \pi(\vect{w}),\pi(\vect{w}')\rangle$, and similarly for $\Lambda_o$. 
Moreover, $\mathcal{E}\perp \mathcal{O}$: 
if $\vect{w}\in \Lambda_e$ and $\vect{w}'\in \Lambda_o$, then $\langle \vect{w},\vect{w}'\rangle \equiv 2^{m-1} + w_2w_4'-w_4w_2' \equiv 2^{m-1}+ \langle \pi( \vect{w}),\pi(\vect{w}')\rangle $.

Note that because $\Lambda$ is primitive the elements of $\mathcal{E}$ are all primitive. 
Again noting $m>1$, we find that $|\mathcal E| \le 2$.
Therefore, if $|\mathcal O|=0$ we are done.

We make a simple observation about $(\Z/2^m\Z)^2$: 
if $\mathcal{A}$ is a symplectic system with $\vect{p}\perp \mathcal{A}$ for a primitive element $\vect{p}$, then $|\mathcal{A}|\le 1$. Indeed, expanding $\vect{p}$ to a symplectic basis, one finds immediately that $\mathcal{A}$ is a symplectic system that must consist of elements parallel to $\vect{p}$, and so cannot contain more than one element.
Thus, if $|\mathcal{E}|=1$ then $|\mathcal{O}|\le 1$, and again we are done. 

Suppose that $|\mathcal{E}|=0$. If $y$ is even then $\vect{z} \perp \mathcal{O}$. Moreover, as $\vect{v}$ is primitive and $m>1$, $z$ must be odd, and so $|\mathcal{O}|\le 1$. If $y$ is odd then $\mathcal{O}\cup \{\vect{z}\}$ is a symplectic system in $(\Z/2^m\Z)^2$, and hence $|\mathcal{O}|\le 2$. 

Finally, suppose that $\mathcal E = \{\vect{v}_1,\vect{v}_2\}$ and $\vect{v}_3\in \mathcal{O}$. 
If $y$ is even, $\vect{z}$ is primitive and $\langle \vect{z},\vect{v}_i \rangle \equiv2^{m-1}$ for $i=1,2$, and so $\mathcal{E}\cup\{ \vect{z}\}$ forms a primitive symplectic system of size three in $(\Z/2^m\Z)^2$, an impossibility.
If $y$ is odd, then $\langle \vect{z},\vect{v}_1\rangle \equiv \langle \vect{z},\vect{v}_2 \rangle \equiv 0$ while $\langle \vect{z},\vect{v}_3\rangle \equiv 2^{m-1}$. On the other hand, $\mathcal{E}$ is primitive, so $\vect{v}_1$ and $\vect{v}_2$ form a basis for $(\Z/2^m\Z)^2$. Writing $\vect{v}_3$ in terms of $\vect{v}_1$ and $\vect{v}_2$, this implies that $ \langle \vect{z},\vect{v}_3 \rangle \equiv 0$, a contradiction.
\end{proof}

\section{A short proof of the upper bound \texorpdfstring{$2g+1$ \'a la \cite{MRT14}}{2g+1 a la MRT14}}
\label{sec:simple proof}

By lifting to $\Z$, Proposition~\ref{prop:vectors} can be rephrased as follows:

\begin{proposition}
\label{prop:Josh}
Let $\Lambda \subset  \Z ^{2g}$ satisfy $\langle \vect{v},\vect{w} \rangle \equiv2^{m-1} \mod 2^m$ for all distinct $\vect{v},\vect{w} \in\Lambda$. Then $|\Lambda|\le 2g+1$.
\end{proposition}

\begin{proof}
We make the preliminary observation that if $\{\vect{v}_1,\ldots,\vect{v}_n\} \subset \Lambda$ satisfies the linear dependence
\begin{equation}
\label{lin dep}
a_1\vect{v}_1 + \ldots + a_n\vect{v}_n =0~,
\end{equation}
then we must have either that $\{a_i\} \subset 2\Z$ or $\{a_i\}\subset \Z\setminus2\Z$: if $a_i$ were even and $a_j$ were odd, then
\begin{align*}
0 &= \langle \vect{v}_i, a_1\vect{v}_1 + \ldots + a_n \vect{v}_n \rangle \equiv \sum_{\ell \ne i}\  2^{m-1} a_\ell  \mod 2^m~, \text{ and}\\
0 &= \langle \vect{v}_j, a_1 \vect{v}_1 + \ldots + a_n\vect{v}_n \rangle \equiv \sum_{\ell\ne i,j} 2^{m-1} a_\ell \mod 2^m~.
\end{align*}
Together these imply that $2^{m-1}a_j \equiv 0 \mod 2^m$, contradicting the assumption that $a_j$ is odd. 
If $a_i\ne 0$ for some $a_i$ in \eqref{lin dep}, then we may divide \eqref{lin dep} by $\gcd\{a_i:a_i\ne 0\}$, and in the resulting dependence we would find an odd coefficient.
Consequently, any nontrivial linear dependence for a symplectic system, as in \eqref{lin dep}, must have only nonzero coefficients.
This easily implies the desired bound: if $|\Lambda| \ge 2g+2$, then any subset $\{\vect{v}_1,\ldots,\vect{v}_{2g+1}\}$ would determine a linear dependence for $\Lambda$ with some coefficient equal to zero.
\end{proof}

\begin{remark}
We are grateful to Josh Greene for his help with this short proof.
\end{remark}

\bigskip
\section{Algebraic Systems of Curves} \label{sec:curves}

By Proposition~\ref{prop:vectors}, any subset $\Lambda \subset \Z^{2g}$ with $ \langle \vect{v}, \vect{w} \rangle \equiv 2^{ m - 1 } \pmod{ 2^{ m } } $ for every $\vect{v},\vect{w}\in\Lambda$ has $|\Lambda| \le 2 g + 1 $.
If, moreover, $m>1$, $g\le 2$, and $\Lambda$ consists of primitive elements, then by Lemma~\ref{lem:low genus} we have $|\Lambda|\le 2g$.
Therefore to prove Theorem~\ref{thm:curves}, it remains to construct symplectic $k$-systems in $\Z^{ 2 g } $ of the appropriate sizes.
As for Corollary~\ref{cor:systems}, any collection of primitive $\Z$-homology classes can be lifted to oriented simple curves on $S$ whose pairwise algebraic intersection numbers are given by the symplectic pairing of the classes \cite[Prop.~6.2]{FM12}. 
Therefore Corollary~\ref{cor:systems} follows from Theorem~\ref{thm:curves}. 

We will now describe two different constructions of large symplectic systems $\Lambda$. 

\begin{construction}
\label{constr:first}
For the first, we will construct a large algebraic $k$-system of simple curves $\Delta_0$. Choosing orientations arbitrarily for the curves in $\Delta_0$, the induced homology classes form a symplectic $k$-system $\Lambda$ with $|\Lambda|=|\Delta_0|$: any pair of distinct curves in $\Delta_0$ algebraically intersect $\pm k$ times (in fact, in our construction they will intersect geometrically $k$ times), and so determine $\Z$-homology classes with symplectic pairing $\pm k$. When $k$ is odd (i.e.~$m=1$) we find one more curve to add to $\Delta_0$; when $m>1$ this curve is not primitive, and we will have to work slightly harder to locate a primitive algebraic $k$-system (see Construction~\ref{constr:second} below). 
Evidently, a primitive algebraic $k$-system of size $2g+1$ can exist only when $m=1$ or $g\ge3$ by Lemma~\ref{lem:low genus}.

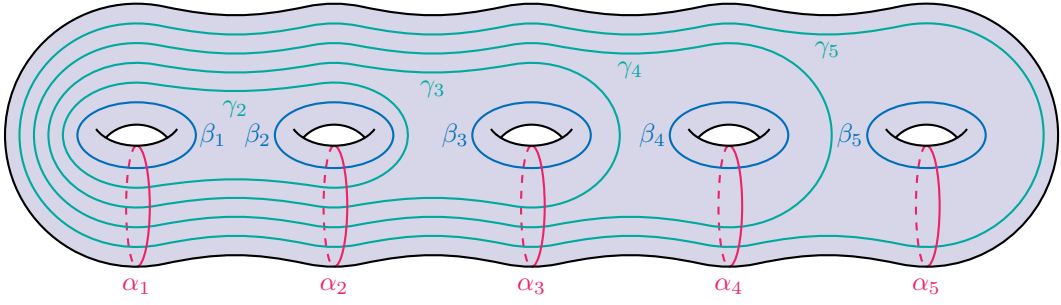
\begin{figure}
\begin{tikzpicture}[ scale = 1.75 ]
\input{picTikz}
\end{tikzpicture}
\caption{Building an algebraic $k$-system on a surface of genus $g$.}
\label{pic:curves}
\end{figure}
Let $ \tau_{ \beta }^{ k }( \alpha ) $ be the $ k $-fold Dehn twist of a curve $ \alpha $ around the simple curve $ \beta $.
Let $ \alpha_{ 1 } , \dotsc , \alpha_{ g } $, $ \beta_{ 1 } , \dotsc , \beta_{ g } $, and $ \gamma_{ 2 } , \dotsc , \gamma_{ g } $ be the curves drawn on $ S $ as in Figure~\ref{pic:curves} when $g=5$. 
Let $ \delta_{ 1 } \coloneqq \alpha_{ 1 } $. 
For each index  $ 2 \leq i \leq g $, let $ \delta_{ i } \coloneqq \tau_{ \gamma_{ i } }^{ k }( \alpha_{ i } ) $, and for each index $1\leq i \leq g$, let $\epsilon_i \coloneqq \tau_{\beta_i}^k(\delta_i)$.
Finally let $ \Delta_0 \coloneqq \{ \delta_{ 1 } , \dotsc , \delta_{ g } , \epsilon_{ 1 } , \dotsc , \epsilon_{ g } \} $, and let $\Lambda_0$ be the set of $\Z$-homology classes of (chosen orientations for) the curves in $\Delta_0$.

Now we claim that elements in $\Lambda_0$ pairwise have symplectic pairing $\pm k$. Indeed, up to the ambiguity of orientation,
$\pm[\delta_i] = [\alpha_i] + k ( [\beta_1]+\ldots+[\beta_i])$ and $\pm[\epsilon_i] = [\delta_i] + k[\beta_i]$, and it is straightforward to compute that $|\langle [\epsilon_i],[\delta_j]\rangle| = k$ for all $1\le i,j\le g$, and $|\langle [\epsilon_i] , [\epsilon_j] \rangle| = | \langle [\delta_i],[\delta_j]\rangle| =k$ for all $1\le i< j \le g$.

When $k=2n+1$ is odd (in particular, when $m=1$), we let $\epsilon_0\coloneqq \tau_{\beta_1}^n \circ \tau_{\alpha_1}^2(\beta_1) $ and $\Delta = \Delta_0 \cup \{\epsilon_0\}$. 
Note that the $\Z$-homology class of $\epsilon_0$ is given by
\[
\vect{v}_0\coloneqq 2[\alpha_1] + k[\beta_1]~,
\]
and the elements of $\Lambda = \Lambda_0 \cup \{ \vect{v}_0 \}$ pairwise have symplectic pairing $\pm k$: indeed, we have $|\langle 2[\alpha_1]+k[\beta_1],[\delta_i]\rangle|=|\langle 2[\alpha_1]+k[\beta_1],[\epsilon_i]\rangle| =k$ for all $1\le i\le g$.
Of course, when $m>1$ the last element $\vect{v}_0$ is non-primitive.
\end{construction}

\begin{construction}
\label{constr:second}
For the second construction, we will build primitive symplectic $k$-systems in $\Z^{2g}$ of size $2g+1$ when $g\ge 3$ (and with no restriction on $m$). 
Roughly speaking, this amounts to a careful choice for $g=3$, followed by a `direct summing'-like operation to build such systems in higher genus. 

Consider the set of seven $\Z$-homology classes $\{\vect{v}_0,\ldots,\vect{v}_6\}$, where $\vect{v}_0=[\alpha_1]$ and
\begin{align*}
\vect{v}_1 & = k[\beta_1] + [\beta_3] \\
\vect{v}_2 & = k[\beta_1] +[\alpha_2]+ k[\beta_2]+ k[\alpha_3]  \\
\vect{v}_3 & = k[\beta_1] +[\alpha_2]+  k[\beta_2]+k[\alpha_3]  + [\beta_3] \\
\vect{v}_4 & = [\alpha_1]+k[\beta_1] + [\alpha_2] +  [\beta_3] \\
\vect{v}_5 & = [\alpha_1] +k[\beta_1]+ k[\beta_2] + [\beta_3] \\
\vect{v}_6 &= [\alpha_1]+k[\beta_1]+[\alpha_2] + k[\beta_2]+2[\beta_3]~.
\end{align*}
Evidently $\{\vect{v}_0,\ldots,\vect{v}_6\}$ is a primitive symplectic $k$-system in $\Z^{2g}$ that lies in the subspace $\langle [\alpha_i],[\beta_j]\rangle_{i,j=1}^3$.

Now, given a primitive symplectic $k$-system $\{\vect{w}_0,\ldots,\vect{w}_r\} \subset \langle [\alpha_i],[\beta_j]\rangle_{i,j=1}^r$ with $r<g$, we may let 
\begin{align*}
\vect{w}_r' &= \vect{w}_r + [\alpha_{r+1}] \\
\vect{w}_{r+1} &= \vect{w}_r + k[\beta_{r+1}] \\
\vect{w}_{r+2} &= \vect{w}_r + [\alpha_{r+1}] + k[\beta_{r+1}]~.
\end{align*}
It is evident that $\{\vect{w}_0,\ldots,\vect{w}_{r-1},\vect{w}_r',\vect{w}_{r+1},\vect{w}_{r+2}\} \subset \langle [\alpha_i],[\beta_j]\rangle_{i,j=1}^{r+1}$ is again a primitive symplectic $k$-system. In this way, $\{\vect{v}_0,\ldots,\vect{v}_6\}$ gives rise to a primitive symplectic $k$-system of size $2g+1$ for any $g\ge 3$.
\end{construction}

\begin{remark}
\label{rem:even odd}
In order to observe why the upper bound of $2g$ for primitive symplectic $k$-systems in $\Z^{2g}$ fails when $g= 3$, the reader should compare Construction~\ref{constr:second} with the proof of Lemma~\ref{lem:low genus}. 
In the language of that proof, the primitive symplectic $k$-system $\{\vect v_0,\ldots,\vect v_6\}\subset\Z^6$ has $|\mathcal{E}|=2$ and $|\mathcal{O}|=3$; each is a primitive symplectic $k$-system in $\Z^4$ of size at most four, and $\mathcal{E}\perp \mathcal{O}$, but $|\mathcal{E}|+|\mathcal{O}|>4$.
\end{remark}

\begin{remark}
\label{rem:realization}
We are unable to decide whether Construction~\ref{constr:second} lifts to a system of simple curves whose pairwise geometric intersection numbers are equal to $k$, even for the case $k=2$ and $g=3$.
Initial attempts were only successful with lifting six of the curves, demonstrating a collection of curves that algebraically and geometrically intersect twice of size $2g$. 
Thus we leave as a challenge for the reader: on a closed surface of genus three, do there exist seven simple curves pairwise algebraically and geometrically intersecting twice?
\end{remark}

\bigskip
\section{Many curves intersecting exactly twice}
\label{sec:many curves}
Finally, we note that the assumption in Corollary~\ref{cor:systems} that curves \emph{algebraically} pair to $\pm k$ is quite essential. By contrast, the following construction demonstrates that the maximum size of a collection of curves that pairwise \emph{geometrically} intersect twice is at least quadratic in the genus.

\begin{construction}
\label{constr:third}
Consider a regular Euclidean $g$-gon $P$ with vertices labelled counterclockwise $v_1,\ldots,v_g$ and edges $e_1,\ldots,e_g$, where $e_i$ is incident to $v_{i-1}$ and $v_i$. We may double $P$ along its boundary to obtain a sphere with $g$ marked points, consisting of two $g$-gons, $P$ and $\overline{P}$. 
When small disks centered at the marked points are cut out and one-holed tori are glued in, we obtain a closed surface of genus $g$.

We now consider a family of curves determined by a sequence of four edges $(e_{i_1},e_{i_2},e_{i_3},e_{i_4})$, for distinct indices $i_1,i_2,i_3,i_4$.
The choice of a pair of edges in $P$ (resp.~$\overline{P}$) determines a simple arc in $P$ (resp.~$\overline{P}$) that connects the midpoints of the two edges. The union of the $(e_{i_1},e_{i_2})$- and $(e_{i_3},e_{i_4})$-arcs in $P$ and the $(e_{i_2},e_{i_3})$- and the $(e_{i_4},e_{i_1})$-arcs in $\overline{P}$ determines a closed curve, which we denote $\beta(i_1,i_2,i_3,i_4)$. 
Note that $\beta(i_1,i_2,i_3,i_4)$ is simple precisely when $(i_1,i_2,i_3,i_4)$ is consistent with the cyclic order $(1,\ldots,g)$.
(There is ambiguity in that $\beta(i_1,i_2,i_3,i_4)=\beta(i_3,i_4,i_1,i_2)$, but this won't affect what follows.)

Now choose $k\ge 1$. We indicate $\alpha(i,j)\coloneqq\beta(1,i,j,j+k-i+1)$.  By construction, $\alpha(i,j)$ partitions the vertices into $\left([1,i-1]\cup[j,j+k-i]\right) $ and $ \left( [i,j-1]\cup[j+k-i+1,g]\right)$, two sets with $k$ and $g-k$ vertices, respectively. If $\alpha(i,j) $ and $\alpha(i',j')$ were freely homotopic, they would determine the same partition, and hence $i=i'$ and $j=j'$.
Moreover, note that $\alpha(i,j)$ is simple provided that $1<i<j<j+k-i+1\le g$, i.e.~when $2\le i \le k$ and $i+1\le j \le i+g-k-1$. 
See Figure~\ref{pic:polygons} for some examples.

\begin{figure}
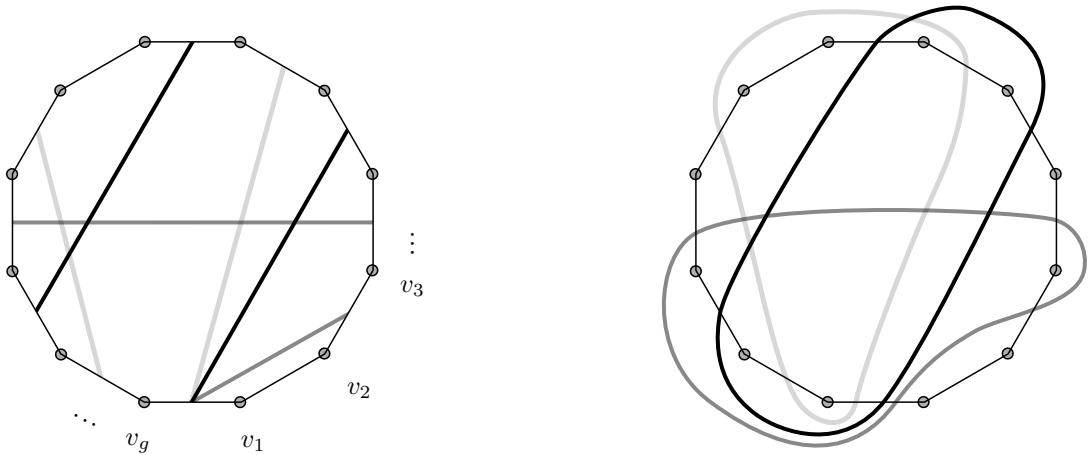

\begin{minipage}{.45\textwidth}
\centering
\begin{lpic}{polygon1(6cm)}
	\lbl[]{310,5;$v_1$}
	\lbl[]{430,65;$v_2$}
	\lbl[]{490,180;$v_3$}
	\lbl[]{490,220;$\cdot$}
	\lbl[]{490,230;$\cdot$}
	\lbl[]{490,240;$\cdot$}
	\lbl[]{180,5;$v_g$}
	\lbl[]{110,35;$\cdot$}
	\lbl[]{120,30;$\cdot$}
	\lbl[]{130,25;$\cdot$}
\end{lpic}
\end{minipage}\hfill
\begin{minipage}{.45\textwidth}
\centering
\begin{lpic}{polygon2(6cm)}
\end{lpic}
\end{minipage}
\caption{Construction~\ref{constr:third} in the case that $g=12$ and $k=8$. Representatives for $\alpha(3,4)$, $\alpha(5,7)$, and $\alpha(6,9)$ are shown on the right in minimal position.}
\label{pic:polygons}
\end{figure}

Any pair of curves $\beta(i_1,i_2,i_3,i_4)$ and $\beta(j_1,j_2,j_3,j_4)$ intersect at most four times. Moreover, if $i_1=j_1$, then examination reveals that they intersect fewer than four times. Since these curves are contained in a planar subsurface, they must intersect an even number of times. We conclude that $\alpha(i_1,j_1)$ and $\alpha(i_2,j_2)$ are either disjoint or intersect twice. On the other hand, in this planar subsurface both $\alpha(i_1,j_1)$ and $\alpha(i_2,j_2)$ bound a disk containing $k$ marked vertices, one of which is $v_1$. If $\alpha(i_1,j_1)$ and $\alpha(i_2,j_2)$ did admit disjoint representatives, by the classification of surfaces we would find that they cobound an annulus, and hence $(i_1,j_1)=(i_2,j_2)$.  
Therefore, we find that
\[
\mathcal{A}_k \coloneqq \left\{ \alpha(i,j) : 2\le i \le k \text{ and } i+1\le j \le i+g-k-1 \right\}
\]
is a collection of curves that pairwise geometrically intersect exactly twice.
Note that
\[
|\mathcal{A}_k| = (k-1)(g-k-1) = (k-1)g - k^2 + 1~.
\]
If we let $k=\lfloor g/2 \rfloor$ we find:
\end{construction}

\begin{proposition}
\label{prop:many curves}
There is a curve collection $\Delta$ on the closed oriented surface of genus $g$ so that each pair of curves geometrically intersects exactly twice, and so that $|\Delta| >\textstyle \frac14 g^2 -g$.
\end{proposition}

Of course, the curve collection above has pairwise algebraic intersection numbers equal to zero, so it is somewhat `orthogonal' to Corollary~\ref{cor:systems}. The size of any curve system with pairwise intersections at most two is known to be $O(g^3 \log \,g)$ by \cite[Thm.~3]{Gre18b}, so the following remains salient:

\begin{problem*}
\label{q:exact 2-system}
Determine the growth rate of the maximum size of a collection of curves with pairwise geometric intersection numbers equal to two and pairwise algebraic intersection numbers equal to zero. Is it closer to $g^2$ or $g^3\log\, g$?
What about $k>2$?
\end{problem*}

\bigskip

\bibliographystyle{alpha}

\vspace{1cm}

\end{document}

%% file: picTikz.tex
\fill[ Periwinkle , opacity = 0.3 ]
	( -1.0 , 0 )
		arc[ radius = 1 , start angle = 180 , end angle = 270 + 12 ]
		to[ out = 12 , in = 180 - 12 ] ( { 1.5 - 0.208 } , -0.978 )
		arc[ radius = 1 , start angle = 270 - 12 , end angle = 270 + 12 ]
		to[ out = 12 , in = 180 - 12 ] ( { 3.0 - 0.208 } , -0.978 )
		arc[ radius = 1 , start angle = 270 - 12 , end angle = 270 + 12 ]
		to[ out = 12 , in = 180 - 12 ] ( { 4.5 - 0.208 } , -0.978 )
		arc[ radius = 1 , start angle = 270 - 12 , end angle = 270 + 12 ]
		to[ out = 12 , in = 180 - 12 ] ( { 6.0 - 0.208 } , -0.978 )
		arc[ radius = 1 , start angle = 270 - 12 - 360 , end angle = 90 + 12 ]
		to[ out = 180 + 12 , in = -12 ] ( { 4.5 + 0.208 } , 0.978 )
		arc[ radius = 1 , start angle = 90 - 12 , end angle = 90 + 12 ]
		to[ out = 180 + 12 , in = -12 ] ( { 3.0 + 0.208 } , 0.978 )
		arc[ radius = 1 , start angle = 90 - 12 , end angle = 90 + 12 ]
		to[ out = 180 + 12 , in = -12 ] ( { 1.5 + 0.208 } , 0.978 )
		arc[ radius = 1 , start angle = 90 - 12 , end angle = 90 + 12 ]
		to[ out = 180 + 12 , in = -12 ] ( { 0.0 + 0.208 } , 0.978 )
		arc[ radius = 1 , start angle = 90 - 12 , end angle = 180 ] ;

\draw[ WildStrawberry , thick , dashed ]
	( 0 , -0.08 ) arc[ x radius = 0.08 , y radius = 0.92 / 2 , start angle = 90 , end angle = 270 ]
	( 1.5 , -0.08 ) arc[ x radius = 0.08 , y radius = 0.92 / 2 , start angle = 90 , end angle = 270 ]
	( 3 , -0.08 ) arc[ x radius = 0.08 , y radius = 0.92 / 2 , start angle = 90 , end angle = 270 ]
	( 4.5 , -0.08 ) arc[ x radius = 0.08 , y radius = 0.92 / 2 , start angle = 90 , end angle = 270 ]
	( 6 , -0.08 ) arc[ x radius = 0.08 , y radius = 0.92 / 2 , start angle = 90 , end angle = 270 ] ;

\draw[ NavyBlue , thick ]
	( 0 , 0 ) circle[ x radius = 0.45 , y radius = 0.25 ]
	( 0.58 , 0 ) node {$ \beta_{ 1 } $}
	
	( 1.5 , 0 ) circle[ x radius = 0.45 , y radius = 0.25 ]
	( 0.92 , 0 ) node {$ \beta_{ 2 } $}
	
	( 3.0 , 0 ) circle[ x radius = 0.45 , y radius = 0.25 ]
	( 2.42 , 0 ) node {$ \beta_{ 3 } $}
	
	( 4.5 , 0 ) circle[ x radius = 0.45 , y radius = 0.25 ]
	( 3.92 , 0 ) node {$ \beta_{ 4 } $}
	
	( 6.0 , 0 ) circle[ x radius = 0.45 , y radius = 0.25 ]
	( 5.42 , 0 ) node {$ \beta_{ 5 } $}
	;

\draw[ Emerald , thick ]
	( 0 - 0.45 - 1 * 0.11 , 0 )
		arc[ x radius = 0.45 + 1 * 0.11 , y radius = 0.25 + 1 * 0.15 , start angle = 180 , end angle = 270 + 12 ]
		to[ out = 8.633 , in = 180 - 8.633 ] ( { 1.5 - 0.116 } , -0.391 )
		arc[ x radius = 0.45 + 1 * 0.11 , y radius = 0.25 + 1 * 0.15 , start angle = 270 - 12 , end angle = 360 + 90 + 12 ]
		to[ out = 180 + 8.633 , in = -8.633 ] ( { 0.0 + 0.116 } , 0.391 )
		arc[ x radius = 0.45 + 1 * 0.11 , y radius = 0.25 + 1 * 0.15 , start angle = 90 - 12 , end angle = 180 ]
	( 0.75 + 0 * 1.5 , 0.2 + 0 * 0.15 ) node {$ \gamma_{ 2 } $}
	
	( 0 - 0.45 - 2 * 0.11 , 0 )
		arc[ x radius = 0.45 + 2 * 0.11 , y radius = 0.25 + 2 * 0.15 , start angle = 180 , end angle = 270 + 12 ]
		to[ out = 9.898 , in = 180 - 9.898 ] ( { 1.5 - 0.139 } , -0.538 )
		arc[ x radius = 0.45 + 2 * 0.11 , y radius = 0.25 + 2 * 0.15 , start angle = 270 - 12 , end angle = 270 + 12 ]
		to[ out = 9.898 , in = 180 - 9.898 ] ( { 3.0 - 0.139 } , -0.538 )
		arc[ x radius = 0.45 + 2 * 0.11 , y radius = 0.25 + 2 * 0.15 , start angle = 270 - 12 , end angle = 360 + 90 + 12 ]
		to[ out = 180 + 9.898 , in = -9.898 ] ( { 1.5 + 0.139 } , 0.538 )
		arc[ x radius = 0.45 + 2 * 0.11 , y radius = 0.25 + 2 * 0.15 , start angle = 90 - 12 , end angle = 90 + 12 ]
		to[ out = 180 + 9.898 , in = -9.898 ] ( { 0.0 + 0.139 } , 0.538 )
		arc[ x radius = 0.45 + 2 * 0.11 , y radius = 0.25 + 2 * 0.15 , start angle = 90 - 12 , end angle = 180 ]
	( 0.75 + 1 * 1.5 , 0.2 + 1 * 0.15 ) node {$ \gamma_{ 3 } $}

	( 0 - 0.45 - 3 * 0.11 , 0 )
		arc[ x radius = 0.45 + 3 * 0.11 , y radius = 0.25 + 3 * 0.15 , start angle = 180 , end angle = 270 + 12 ]
		to[ out = 10.8 , in = 180 - 10.8 ] ( { 1.5 - 0.162 } , -0.685 )
		arc[ x radius = 0.45 + 3 * 0.11 , y radius = 0.25 + 3 * 0.15 , start angle = 270 - 12 , end angle = 270 + 12 ]
		to[ out = 10.8 , in = 180 - 10.8 ] ( { 3.0 - 0.162 } , -0.685 )
		arc[ x radius = 0.45 + 3 * 0.11 , y radius = 0.25 + 3 * 0.15 , start angle = 270 - 12 , end angle = 270 + 12 ]
		to[ out = 10.8 , in = 180 - 10.8 ] ( { 4.5 - 0.162 } , -0.685 )
		arc[ x radius = 0.45 + 3 * 0.11 , y radius = 0.25 + 3 * 0.15 , start angle = 270 - 12 , end angle = 360 + 90 + 12 ]
		to[ out = 180 + 10.8 , in = -10.8 ] ( { 3.0 + 0.162 } , 0.685 )
		arc[ x radius = 0.45 + 3 * 0.11 , y radius = 0.25 + 3 * 0.15 , start angle = 90 - 12 , end angle = 90 + 12 ]
		to[ out = 180 + 10.8 , in = -10.8 ] ( { 1.5 + 0.162 } , 0.685 )
		arc[ x radius = 0.45 + 3 * 0.11 , y radius = 0.25 + 3 * 0.15 , start angle = 90 - 12 , end angle = 90 + 12 ]
		to[ out = 180 + 10.8 , in = -10.8 ] ( { 0.0 + 0.162 } , 0.685 )
		arc[ x radius = 0.45 + 3 * 0.11 , y radius = 0.25 + 3 * 0.15 , start angle = 90 - 12 , end angle = 180 ]
	( 0.75 + 2 * 1.5 , 0.2 + 2 * 0.15 ) node {$ \gamma_{ 4 } $}
	
	( 0 - 0.45 - 4 * 0.11 , 0 )
		arc[ x radius = 0.45 + 4 * 0.11 , y radius = 0.25 + 4 * 0.15 , start angle = 180 , end angle = 270 + 12 ]
		to[ out = 11.475 , in = 180 - 11.475 ] ( { 1.5 - 0.185 } , -0.831 )
		arc[ x radius = 0.45 + 4 * 0.11 , y radius = 0.25 + 4 * 0.15 , start angle = 270 - 12 , end angle = 270 + 12 ]
		to[ out = 11.475 , in = 180 - 11.475 ] ( { 3.0 - 0.185 } , -0.831 )
		arc[ x radius = 0.45 + 4 * 0.11 , y radius = 0.25 + 4 * 0.15 , start angle = 270 - 12 , end angle = 270 + 12 ]
		to[ out = 11.475 , in = 180 - 11.475 ] ( { 4.5 - 0.185 } , -0.831 )
		arc[ x radius = 0.45 + 4 * 0.11 , y radius = 0.25 + 4 * 0.15 , start angle = 270 - 12 , end angle = 270 + 12 ]
		to[ out = 11.475 , in = 180 - 11.475 ] ( { 6.0 - 0.185 } , -0.831 )
		arc[ x radius = 0.45 + 4 * 0.11 , y radius = 0.25 + 4 * 0.15 , start angle = 270 - 12 , end angle = 360 + 90 + 12 ]
		to[ out = 180 + 11.475 , in = -11.475 ] ( { 4.5 + 0.185 } , 0.831 )
		arc[ x radius = 0.45 + 4 * 0.11 , y radius = 0.25 + 4 * 0.15 , start angle = 90 - 12 , end angle = 90 + 12 ]
		to[ out = 180 + 11.475 , in = -11.475 ] ( { 3.0 + 0.185 } , 0.831 )
		arc[ x radius = 0.45 + 4 * 0.11 , y radius = 0.25 + 4 * 0.15 , start angle = 90 - 12 , end angle = 90 + 12 ]
		to[ out = 180 + 11.475 , in = -11.475 ] ( { 1.5 + 0.185 } , 0.831 )
		arc[ x radius = 0.45 + 4 * 0.11 , y radius = 0.25 + 4 * 0.15 , start angle = 90 - 12 , end angle = 90 + 12 ]
		to[ out = 180 + 11.475 , in = -11.475 ] ( { 0.0 + 0.185 } , 0.831 )
		arc[ x radius = 0.45 + 4 * 0.11 , y radius = 0.25 + 4 * 0.15 , start angle = 90 - 12 , end angle = 180 ]
	( 0.75 + 3 * 1.5 , 0.2 + 3 * 0.15 ) node {$ \gamma_{ 5 } $} ;

\draw[ WildStrawberry , thick ]
	( 0 , -0.08 ) arc[ x radius = 0.1 , y radius = 0.92 / 2 , start angle = 90 , end angle = -90 ]
	( 0 , -1.15 ) node {$ \alpha_{ 1 } $}
	
	( 1.5 , -0.08 ) arc[ x radius = 0.1 , y radius = 0.92 / 2 , start angle = 90 , end angle = -90 ]
	( 1.5 , -1.15 ) node {$ \alpha_{ 2 } $}
	
	( 3 , -0.08 ) arc[ x radius = 0.1 , y radius = 0.92 / 2 , start angle = 90 , end angle = -90 ]
	( 3.0 , -1.15 ) node {$ \alpha_{ 3 } $}
	
	( 4.5 , -0.08 ) arc[ x radius = 0.1 , y radius = 0.92 / 2 , start angle = 90 , end angle = -90 ]
	( 4.5 , -1.15 ) node {$ \alpha_{ 4 } $}
	
	( 6 , -0.08 ) arc[ x radius = 0.1 , y radius = 0.92 / 2 , start angle = 90 , end angle = -90 ]
	( 6.0 , -1.15 ) node {$ \alpha_{ 5 } $}
	;

\fill[ white , rotate around = { 90.0 : ( 0 , 0 ) } ]
	( -0.08 , 0 )
		arc[ x radius = 0.22 , y radius = 0.34 , start angle = 180 , end angle = 136 ]
		arc[ x radius = 0.266 , y radius = 0.3 , start angle = 50 , end angle = -50 ]
		arc[ x radius = 0.22 , y radius = 0.34 , start angle = 224 , end angle = 180 ] ;

\draw[ thick , rotate around = { 90.0 : ( 0 , 0 ) } ]
	( -0.08 , 0 )
		arc[ x radius = 0.22 , y radius = 0.34 , start angle = 180 , end angle = 115 ]
	( -0.08 , 0 )
		arc[ x radius = 0.22 , y radius = 0.34 , start angle = 180 , end angle = 245 ]
	( 0.08 , 0 )
		arc[ x radius = 0.266 , y radius = 0.3 , start angle = 0 , end angle = 50 ]
	( 0.08 , 0 )
		arc[ x radius = 0.266 , y radius = 0.3 , start angle = 0 , end angle = -50 ] ;

\fill[ white , rotate around = { 90.0 : ( 1.5 , 0 ) } ]
	( 1.5 - 0.08 , 0 )
		arc[ x radius = 0.22 , y radius = 0.34 , start angle = 180 , end angle = 136 ]
		arc[ x radius = 0.266 , y radius = 0.3 , start angle = 50 , end angle = -50 ]
		arc[ x radius = 0.22 , y radius = 0.34 , start angle = 224 , end angle = 180 ] ;

\draw[ thick , rotate around = { 90.0 : ( 1.5 , 0 ) } ]
	( 1.5 - 0.08 , 0 )
		arc[ x radius = 0.22 , y radius = 0.34 , start angle = 180 , end angle = 115 ]
	( 1.5 - 0.08 , 0 )
		arc[ x radius = 0.22 , y radius = 0.34 , start angle = 180 , end angle = 245 ]
	( 1.5 + 0.08 , 0 )
		arc[ x radius = 0.266 , y radius = 0.3 , start angle = 0 , end angle = 50 ]
	( 1.5 + 0.08 , 0 )
		arc[ x radius = 0.266 , y radius = 0.3 , start angle = 0 , end angle = -50 ] ;

\fill[ white , rotate around = { 90.0 : ( 3.0 , 0 ) } ]
	( 3.0 - 0.08 , 0 )
		arc[ x radius = 0.22 , y radius = 0.34 , start angle = 180 , end angle = 136 ]
		arc[ x radius = 0.266 , y radius = 0.3 , start angle = 50 , end angle = -50 ]
		arc[ x radius = 0.22 , y radius = 0.34 , start angle = 224 , end angle = 180 ] ;

\draw[ thick , rotate around = { 90.0 : ( 3.0 , 0 ) } ]
	( 3.0 - 0.08 , 0 )
		arc[ x radius = 0.22 , y radius = 0.34 , start angle = 180 , end angle = 115 ]
	( 3.0 - 0.08 , 0 )
		arc[ x radius = 0.22 , y radius = 0.34 , start angle = 180 , end angle = 245 ]
	( 3.0 + 0.08 , 0 )
		arc[ x radius = 0.266 , y radius = 0.3 , start angle = 0 , end angle = 50 ]
	( 3.0 + 0.08 , 0 )
		arc[ x radius = 0.266 , y radius = 0.3 , start angle = 0 , end angle = -50 ] ;

\fill[ white , rotate around = { 90.0 : ( 4.5 , 0 ) } ]
	( 4.5 - 0.08 , 0 )
		arc[ x radius = 0.22 , y radius = 0.34 , start angle = 180 , end angle = 136 ]
		arc[ x radius = 0.266 , y radius = 0.3 , start angle = 50 , end angle = -50 ]
		arc[ x radius = 0.22 , y radius = 0.34 , start angle = 224 , end angle = 180 ] ;

\draw[ thick , rotate around = { 90.0 : ( 4.5 , 0 ) } ]
	( 4.5 - 0.08 , 0 )
		arc[ x radius = 0.22 , y radius = 0.34 , start angle = 180 , end angle = 115 ]
	( 4.5 - 0.08 , 0 )
		arc[ x radius = 0.22 , y radius = 0.34 , start angle = 180 , end angle = 245 ]
	( 4.5 + 0.08 , 0 )
		arc[ x radius = 0.266 , y radius = 0.3 , start angle = 0 , end angle = 50 ]
	( 4.5 + 0.08 , 0 )
		arc[ x radius = 0.266 , y radius = 0.3 , start angle = 0 , end angle = -50 ] ;

\fill[ white , rotate around = { 90.0 : ( 6.0 , 0 ) } ]
	( 6.0 - 0.08 , 0 )
		arc[ x radius = 0.22 , y radius = 0.34 , start angle = 180 , end angle = 136 ]
		arc[ x radius = 0.266 , y radius = 0.3 , start angle = 50 , end angle = -50 ]
		arc[ x radius = 0.22 , y radius = 0.34 , start angle = 224 , end angle = 180 ] ;

\draw[ thick , rotate around = { 90.0 : ( 6.0 , 0 ) } ]
	( 6.0 - 0.08 , 0 )
		arc[ x radius = 0.22 , y radius = 0.34 , start angle = 180 , end angle = 115 ]
	( 6.0 - 0.08 , 0 )
		arc[ x radius = 0.22 , y radius = 0.34 , start angle = 180 , end angle = 245 ]
	( 6.0 + 0.08 , 0 )
		arc[ x radius = 0.266 , y radius = 0.3 , start angle = 0 , end angle = 50 ]
	( 6.0 + 0.08 , 0 )
		arc[ x radius = 0.266 , y radius = 0.3 , start angle = 0 , end angle = -50 ] ;

\draw[ thick ]
	( -1.0 , 0 )
		arc[ radius = 1 , start angle = 180 , end angle = 270 + 12 ]
		to[ out = 12 , in = 180 - 12 ] ( { 1.5 - 0.208 } , -0.978 )
		arc[ radius = 1 , start angle = 270 - 12 , end angle = 270 + 12 ]
		to[ out = 12 , in = 180 - 12 ] ( { 3.0 - 0.208 } , -0.978 )
		arc[ radius = 1 , start angle = 270 - 12 , end angle = 270 + 12 ]
		to[ out = 12 , in = 180 - 12 ] ( { 4.5 - 0.208 } , -0.978 )
		arc[ radius = 1 , start angle = 270 - 12 , end angle = 270 + 12 ]
		to[ out = 12 , in = 180 - 12 ] ( { 6.0 - 0.208 } , -0.978 )
		arc[ radius = 1 , start angle = 270 - 12 - 360 , end angle = 90 + 12 ]
		to[ out = 180 + 12 , in = -12 ] ( { 4.5 + 0.208 } , 0.978 )
		arc[ radius = 1 , start angle = 90 - 12 , end angle = 90 + 12 ]
		to[ out = 180 + 12 , in = -12 ] ( { 3.0 + 0.208 } , 0.978 )
		arc[ radius = 1 , start angle = 90 - 12 , end angle = 90 + 12 ]
		to[ out = 180 + 12 , in = -12 ] ( { 1.5 + 0.208 } , 0.978 )
		arc[ radius = 1 , start angle = 90 - 12 , end angle = 90 + 12 ]
		to[ out = 180 + 12 , in = -12 ] ( { 0.0 + 0.208 } , 0.978 )
		arc[ radius = 1 , start angle = 90 - 12 , end angle = 180 ] ;